\documentclass[final,onefignum,onetabnum]{siuro250106}

\usepackage{lipsum}
\usepackage{amsfonts}
\usepackage{graphicx}
\usepackage{epstopdf}
\usepackage{algorithmic}

\ifpdf
  \DeclareGraphicsExtensions{.eps,.pdf,.png,.jpg}
\else
  \DeclareGraphicsExtensions{.eps}
\fi

\usepackage{enumitem}
\setlist[enumerate]{leftmargin=.5in}
\setlist[itemize]{leftmargin=.5in}

\DeclareMathOperator*{\argmin}{arg\,min}

\theoremstyle{remark}
\newtheorem{remark}{Remark}

\newcommand{\ve}{{\mathbf{e}}}

\newcommand{\vs}{{\mathbf{s}}}

\newcommand{\vx}{{\mathbf{x}}}
\newcommand{\vy}{{\mathbf{y}}}
\newcommand{\vz}{{\mathbf{z}}}

\newcommand{\vA}{{\mathbf{A}}}
\newcommand{\vB}{{\mathbf{B}}}
\newcommand{\vC}{{\mathbf{C}}}
\newcommand{\vD}{{\mathbf{D}}}
\newcommand{\vE}{{\mathbf{E}}}

\newcommand{\vI}{{\mathbf{I}}}

\newcommand{\vL}{{\mathbf{L}}}
\newcommand{\vM}{{\mathbf{M}}}

\newcommand{\vP}{{\mathbf{P}}}

\newcommand{\vV}{{\mathbf{V}}}
\newcommand{\vW}{{\mathbf{W}}}
\newcommand{\vX}{{\mathbf{X}}}
\newcommand{\vY}{{\mathbf{Y}}}

\newcommand{\vone}{{\mathbf{1}}}

\newcommand{\RR}{\mathbb{R}} % real
\newcommand{\proj}{\mathcal{P}}

\newcommand{\Fsf}{\mathsf{F}}

\newcommand{\Rbb}{\mathbb{R}}
\newcommand{\Dcaln}{\mathcal{D}_n}

\definecolor{DarkGreen}{RGB}{0, 90, 0} 
\usepackage[normalem]{ulem}

\headers{Linear Reweighted Regularization Algorithms for Graph Matching Problem}{R. Li and Y. Xu}

\title{Linear Reweighted Regularization Algorithms for Graph Matching Problem}

\author{Rongxuan Li\thanks{Department of Mathematics, The Chinese University of Hong Kong, Shatin, Hong Kong SAR, China
  (\email{1155173725@link.cuhk.edu.hk)}.}
}

\dedication{\small\textit{Project advisor: Prof. Yangyang Xu\thanks{Department of Mathematical Sciences,
Rensselaer Polytechnic Institute, Troy, New York, USA (\email{xuy21@rpi.edu}).}}}

\usepackage{amsopn}

\ifpdf
\hypersetup{
  pdftitle={Sparse Optimization},
  pdfauthor={R.LI, and Y.Xu}
}
\fi

\begin{document}

\maketitle

% REQUIRED
\begin{abstract}
The graph matching problem is a significant special case of the Quadratic Assignment Problem, with extensive applications in pattern recognition, computer vision, protein alignments and related fields. As the problem is NP-hard, relaxation and regularization techniques are frequently employed to improve tractability. However, most existing regularization terms are nonconvex, posing optimization challenges. In this paper, we propose a linear reweighted regularizer framework for solving the relaxed graph matching problem, preserving the convexity of the formulation. By solving a sequence of relaxed problems with the linear reweighted regularization term, one can obtain a sparse solution that, under certain conditions, theoretically aligns with the original graph matching problem's solution. Furthermore, we present a practical version of the algorithm by incorporating the projected gradient method. The proposed framework is applied to synthetic instances, demonstrating promising numerical results.
\end{abstract}

\section{Introduction}
The Quadratic Assignment Problem (QAP), a key optimization problem over permutation matrices, is one of the hardest combinatorial optimization problems and has wide applications in fields such as statistics, facility layout, and chip design~\cite{foggia,drezner2015quadratic}. Specifically, the QAP involves the assignment of a set of facilities to a set of locations in such a way that minimizes a given cost function, and can be expressed as:

\begin{equation} \label{equQAP}
   \min_{\vX \in \Pi_n} \text{tr}(\vA^\top\vX\vB\vX^\top),
\end{equation} 
where $\Pi_n$ is the set of $n$-order permutation matrices, namely, $\Pi_n = 
\{\vX \in \mathbb{R}^{n \times n} \mid \vX \mathbf{e}$ $ = \mathbf{X}^\top \mathbf{e} = \ve, \vX_{ij} \in \{0, 1\} \}$, $\vA\in\mathbb{R}^{n \times n}$, $\vB\in\mathbb{R}^{n \times n}$, $\ve\in\mathbb{R}^n$ is a vector of all ones. As the QAP Problem is NP-hard~\cite{sahni1976p}, both vertex-based methods and interior-point methods have been proposed to address this challenging problem. Vertex-based methods~\cite{ahuja2000greedy,frieze1989algorithms, taillard1991robust} update iterates directly from one permutation matrix to another. Alternatively, interior-point methods relax the nonconvex permutation matrix constraint $\Pi_n$ to a doubly stochastic matrix constraint $\Dcaln$, where $\Dcaln = \left\{\vX \in \mathbb{R}^{n \times n} \mid \vX \mathbf{e} = \mathbf{X}^\top \mathbf{e} = \ve, \vX_{ij} \ge0 \right\}$. To enhance sparsity in solutions derived from this relaxation, various regularization techniques have been introduced. Xia~\cite{xia2010efficient} proposed $L_2$ regularization approach, Huang~\cite{huang2008} proposed quartic term regularization, and Jiang et al~\cite{WZW} suggested using $L_p$ norms~\cite{lu2014iterative}. These regularization techniques are generally nonconvex, which can pose computational challenges. While adding a nonconvex term to the relaxed problem may improve solution quality, it does not necessarily lead to greater computational efficiency.

Notice that when $\vX \in \Pi_n$, 
\begin{equation}
\|\vA \vX - \vX \vB\|_{\Fsf}^2 = -2\text{tr}(\vA^\top\vX\vB\vX^\top)+ \|\vA\|_{\Fsf}^2 + \|\vB\|_{\Fsf}^2,
\end{equation}
where \(\|\vA\|_\Fsf = \sqrt{\sum_{i=1}^m \sum_{j=1}^n |a_{ij}|^2}\).
Hence, the QAP problem can be expressed in the form of a graph matching problem.
In this paper, we focus on graph matching problem: 
\begin{equation} \label{equ:grahpmatch}
   \min\limits_{\vX \in \Pi_n} \, \|\vA\vX - \vX\vB\|_{\Fsf}^2.
\end{equation} 
We relax the the permutation matrix constraint $\Pi_n$ of graph matching problem to the doubly stochastic matrix constraint $\Dcaln$ and obtain the relaxed graph matching problem (\ref{equ:relaxgraphmatch}). We first extend nonconvex regularization terms and demonstrate that the relaxed problem (\ref{equ:relaxgraphmatch}), when combined with certain concave regularization term, can achieve equivalence between the relaxed and original graph matching problem when regularization term is sufficiently large. However, incorporating a nonconvex term into the relaxed problem may enhance solution quality but does not necessarily improve computational efficiency. To better utilize the convexity of the relaxed problem which is shown in Proposition~\ref{prop}, we propose a novel algorithmic framework based on linear reweighted regularization to preserve the convexity of the relaxed problem. Inspired by the effectiveness of reweighted $L_1$ minimization in promoting sparsity~\cite{candes2008enhancing}, our approach solves a sequence of relaxed problems. We show that, under certain conditions on the initial guess, the model can achieve equivalence to the original graph matching problem. We design a solver based on the projected gradient method and incorporate line search techniques to enhance computational efficiency. Additionally, we introduce the mathematical formulation of the network alignment problem, illustrating how graph matching can address problems in fields such as social networks and computer vision. Finally, we validate our approach through numerical experiments on synthetic full-rank dense datasets, demonstrating promising results across a variety of instances.

\section{Regularization}
\subsection{Relaxation}
The graph matching problem~\eqref{equ:grahpmatch} can be relaxed to 
an optimization problem with $n~th$ order doubly stochastic matrix constraint $\Dcaln$:
\begin{equation} \label{equ:relaxgraphmatch}
   \min_{\vX \in\Dcaln} \|\vA\vX - \vX\vB\|_{\Fsf}^2.
\end{equation} 
The relaxed problem is continuous optimization problem and has several properties.

\begin{proposition}\label{prop}
Let $f(\vX)= \|\vA\vX-\vX\vB\|_{\Fsf}^2$, then $f(\vX)$ is convex on $\RR^{n\times n}$.
\end{proposition}

\begin{proof}
Define the function $g(\vY) = \|\vY\|_{\Fsf}^2$. Its Gradient and Hessian are given by
$$
\nabla g(\vY) = 2\vY, \quad \nabla^2 g(\vY) = 2\vI.
$$
Since the Hessian $\nabla^2 g(\vY)$ is a constant positive definite matrix, it follows that $g(\vY)$ is convex.
Define the function $h$ as
$$
h(\vX) = \vA  \vX - \vX  \vB.
$$
Since $g(\vY)$ is convex and $h(\vX)$ is a linear transformation, the composition $f(\vX) = g(h(\vX))$ is convex.
\end{proof}

\begin{remark}
By the Birkhoff-von Neumann theorem~\cite{birkhoff1946tres}, we know that $\Dcaln$ is the convex hull of $\Pi_n$, and the set of vertices of $\Dcaln$ is exactly $\Pi_n$. Hence, $f(\vX)$ is convex on $\Dcaln$.
\end{remark}

\begin{proposition}
$f(\vX)$ is Lipschitz continuous on $\Dcaln$ one possible Lipschitz constant $L$ shall be $L=2(\|\vA\|_\Fsf+\|\vB\|_\Fsf)^2$.
\end{proposition}

By carefully choosing regularization term and adding it to~\eqref{equ:relaxgraphmatch}, one can better ensure the sparsity of the solution. More specifically, by adding $h(\vX)$ and $\lambda\ge 0$ results in:

\begin{equation} \label{equ:relaxgraphmatchregular}
   \min_{\vX\in\Dcaln} \|\vA\vX - \vX\vB\|_{\Fsf}^2+\lambda h(\vX).
\end{equation} 

\subsection{Nonconvex Sparsity Regularization Terms}
Now we briefly introduce several nonconvex sparsity regularization terms proposed in literature. Huang \cite{huang2008}
used the quartic term:
\begin{equation}
h(\vX)=\|\vX \odot (\mathbf{1} - \vX) \|_{\Fsf}^2,
\end{equation}
to construct the regularization problem, where  $\odot$ is the Hadamard product and $\mathbf{1} \in \Rbb^{n \times n}$ is the matrix of all ones. Jiang et al~\cite{WZW} proposed the $L_p$ norm term by observing that $\Pi_n$ can be equivalently characterized as $\Pi_n = \Dcaln \cap \{X \mid \|X\|_0 = n \},$ and hence use:
\begin{equation}
h(\vX)=\|\vX + \epsilon \mathbf{1}\|_p^p = \sum_{i =1}^n \sum_{j=1}^n (\vX_{ij} + \epsilon)^p,
\end{equation}
with $0< p < 1$ to continuously approximate the $\|X\|_0$. The aforementioned regularization all considered solving optimization over $\Dcaln$ with a concave regularization term. By a similar proof in Theorem~3.2 in~\cite{WZW}, it is not hard to show for any strongly concave regularization term $h(\vX)$, it holds that $h(\vX)=0$, $\forall~\vX\in\Pi_n$ and $h(\vX)>0$, $\forall~\vX\in\Dcaln/\Pi_n$. Then with a finitely large $\lambda$, the problem \eqref{equ:relaxgraphmatchregular} is equivalent to graph matching problem \eqref{equ:grahpmatch}. For example, $h(\vX)=\log\vX(\vX-\mathbf{1})$ can yield the equivalence.

However, adding nonconvex regularization term $h(\vX)$ would make \eqref{equ:relaxgraphmatchregular} nonconvex which may bring computational intractability when using iterative solvers.

\section{Linear Reweighted Regularization Algorithmic Framework}
Considering the good performance of the linear reweighted regularization in recovering sparse solutions~\cite{candes2008enhancing}, we apply it to relax the graph matching problem. 
\subsection{Linear Regularization Term}
Linear regularization term is given by:
\begin{equation}
h_{\vW}(\vX) = \sum_{i=1}^{n} \sum_{j=1}^{n} \vW_{ij} \vX_{ij},    
\end{equation}
where $\vW_{ij}$ represents a weight for  element $\vX_{ij}$ for each $(i,j)$. 
\subsection{Algorithm Description}
We present an iterative algorithm to solve the graph matching problem by using the linear reweighted regularization term. Specifically, the algorithm solves a sequence of graph matching problems, each added by a linear regularization term whose weight is determined based on the solution obtained in the previous iteration, subject to the doubly stochastic matrix constraint $\Dcaln$. With the reweighted regularization, the solution obtained at each subproblem is expected to  approach a sparse solution. The detailed Algorithmic Framework is shown in Algorithm~\ref{alg:LRRAF}. 
\begin{algorithm}
\caption{Linear Reweighted Regularization Algorithmic Framework}
\label{alg:LRRAF}
\begin{algorithmic}[1]
\small
\STATE\textbf{Input}: $\vA\in\mathbb{R}^{n \times n}$, $\vB\in\mathbb{R}^{n \times n}$, $\tau > 0$, $\epsilon_0 > 0$, and $\lambda_0 > 0$.
\STATE\textbf{Initialization}: $\vX^{(0)} \in \mathbb{R}^{n \times n}$, set $k = 0$, $\vX^{(0)} = \frac{1}{n} \vone_{n} \in \mathcal{D}_n$. 
\newline(We denote $\vone_{n}$ as $n \times n$ matrix of all ones)
\WHILE{not convergent}
    \STATE $\vW_{ij}^{(k)} = \frac{1}{\vX_{ij}^{(k)} + \epsilon_k}$
    \STATE $\vX^{(k+1)} = \argmin\limits_{\vX \in \mathcal{D}_n} f(\vX) + \lambda_k \sum_{i=1}^{n} \sum_{j=1}^{n} \vW_{ij}^{(k)} \vX_{ij}$
    \STATE Choose $\lambda_{k+1} \geq \lambda_k$, $\epsilon_{k+1} \leq \epsilon_k$
    \STATE Let $k \gets k + 1$
\ENDWHILE
\end{algorithmic}
\end{algorithm}

\subsection{Convergence}
We now justify the effectiveness of linear reweighted regularization algorithm in solving the graph matching problem by showing its local convergence property. Specifically, Theorem~\ref{mainthm} shows that when the matrix $\vX^{(k)}$ satisfies a specific criterion, the solution $\vX^{(k+1)}$ obtained in the next iteration by reweighted algorithm will be the same as the globally optimal solution of the graph matching problem. To prove Theorem~\ref{mainthm}, we first establish the following lemma.

\begin{lemma}\label{lemma}
Let $\vX, \vY \in \Dcaln$. Then
$$
\sum_{i=1}^{n}\sum_{j=1}^{n} [\vX_{ij} - \vY_{ij}]_{+} = -\sum_{i=1}^{n}\sum_{j=1}^{n} [\vX_{ij} - \vY_{ij}]_{-} = \frac{1}{2} \sum_{i=1}^{n} \sum_{j=1}^{n} |\vX_{ij} - \vY_{ij}|,
$$
where $[\vX_{ij}]_{+} = \max(\vX_{ij}, 0)$ and $[\vX_{ij}]_{-} = \min(\vX_{ij}, 0)$.
\end{lemma}

\begin{proof}
Since $\vX, \vY \in \Dcaln$, we have 
$$
\sum_{i=1}^{n} \sum_{j=1}^{n} \vX_{ij} = \sum_{i=1}^{n} \sum_{j=1}^{n} \vY_{ij} = n.
$$
Hence,
$$
\sum_{i=1}^{n} \sum_{j=1}^{n} (\vX_{ij} - \vY_{ij}) = 0,
$$
and therefore by $\vX_{ij} - \vY_{ij}=[\vX_{ij} - \vY_{ij}]_{+} + [\vX_{ij} - \vY_{ij}]_{-}$, it follows that
$$
\sum_{i=1}^{n} \sum_{j=1}^{n} [\vX_{ij} - \vY_{ij}]_{+} = - \sum_{i=1}^{n} \sum_{j=1}^{n} [\vX_{ij} - \vY_{ij}]_{-}.
$$
Moreover, it holds that $|\vX_{ij} - \vY_{ij}|=[\vX_{ij} - \vY_{ij}]_{+} - [\vX_{ij} - \vY_{ij}]_{-}$, and thus
$$
\sum_{i=1}^{n} \sum_{j=1}^{n} [\vX_{ij} - \vY_{ij}]_{+} - \sum_{i=1}^{n} \sum_{j=1}^{n} [\vX_{ij} - \vY_{ij}]_{-} = \sum_{i=1}^{n} \sum_{j=1}^{n} |\vX_{ij} - \vY_{ij}|.
$$
Hence, we conclude that
$$
\sum_{i=1}^{n} \sum_{j=1}^{n} [\vX_{ij} - \vY_{ij}]_{+} = - \sum_{i=1}^{n} \sum_{j=1}^{n} [\vX_{ij} - \vY_{ij}]_{-} = \frac{1}{2} \sum_{i=1}^{n} \sum_{j=1}^{n} |\vX_{ij} - \vY_{ij}|.
$$
\end{proof}

\begin{theorem} \label{mainthm}
Let $f(\vX) = \|\vA\vX - \vX\vB\|_{\Fsf}^2$ and $\vX^* = \argmin_{\vX \in \Pi_n} f(\vX)$. If $\|\vX^{(k)} - \vX^*\|_{\Fsf} \leq a$ for some $a \in [0, \frac{1}{2})$, and $\lambda \geq \frac{2(a+\epsilon)(1-a+\epsilon)}{(1-2a)}L$, where $L$ is the Lipschitz constant of $f(\vX)$ on $\Dcaln$, then $\vX^* = \argmin_{\vX \in \Dcaln} f(\vX) + \lambda \sum_{i=1}^n\sum_{j=1}^n\left(\frac{1}{\vX_{ij}^{(k)} + \epsilon}\right)\vX_{ij}.$ 

\begin{proof}
By the $L$-Lipschitz continuity of $f$, it holds $f(\vX) \ge f(\vX^*) - L\|\vX^* - \vX\|_{\Fsf}$ for any $\vX \in \Dcaln$. Thus we have
\begin{align*}
    &f(\vX) + \lambda \sum_{i=1}^{n} \sum_{j=1}^{n} \left( \frac{1}{\vX_{ij}^{(k)} + \epsilon} \right) \vX_{ij}\\
    \geq & f(\vX^*) - L\|\vX^* - \vX\|_{\Fsf} 
    + \lambda \sum_{i=1}^{n} \sum_{j=1}^{n} \left( \frac{1}{\vX_{ij}^{(k)} + \epsilon} \right) \vX^*_{ij} + \lambda \sum_{i=1}^{n} \sum_{j=1}^{n} \left( \frac{1}{\vX_{ij}^{(k)} + \epsilon} \right) (\vX_{ij} - \vX^*_{ij}).
\end{align*}
By $\vX_{ij} - \vY_{ij}=[\vX_{ij} - \vY_{ij}]_{+} + [\vX_{ij} - \vY_{ij}]_{-}$, it follows from the inequality above that
\begin{align}\label{eq:bound2}
    &f(\vX) + \lambda \sum_{i=1}^{n} \sum_{j=1}^{n} \left( \frac{1}{\vX_{ij}^{(k)} + \epsilon} \right) \vX_{ij}\nonumber\\
    \geq & f(\vX^*) - L\|\vX^* - \vX\|_{\Fsf} 
    + \lambda \sum_{i=1}^{n} \sum_{j=1}^{n} \left( \frac{1}{\vX_{ij}^{(k)} + \epsilon} \right) \vX^*_{ij} \\
    &\quad + \lambda \sum_{i=1}^{n} \sum_{j=1}^{n} \left( \frac{1}{\vX_{ij}^{(k)}+\epsilon} \right) [\vX_{ij} - \vX^*_{ij}]_{+}  + \lambda \sum_{i=1}^{n} \sum_{j=1}^{n} \left( \frac{1}{\vX_{ij}^{(k)}+\epsilon} \right) [\vX_{ij} - \vX^*_{ij}]_{-}. \nonumber
\end{align}
Notice that since $\vX^*$ is a permutation matrix, $[\vX_{ij} - \vX^*_{ij}]_{+}$ is positive only if $\vX^*_{ij}=0$. In this case, it must hold that $\vX_{ij}^{(k)}\le a$ by the condition $\|\vX^{(k)} - \vX^*\|_{\Fsf} \leq a $ and thus $\frac{1}{\vX_{ij}^{(k)}+\epsilon}\ge \frac{1}{a+\epsilon}$. Similarly, $[\vX_{ij} - \vX^*_{ij}]_{-}$ is negative only if $\vX^*_{ij}=1$. In this case, $\vX_{ij}^{(k)}\ge 1-a$, and $\frac{1}{\vX_{ij}^{(k)}+\epsilon}\le \frac{1}{1-a+\epsilon}$. Hence,
\begin{align*}
    &\lambda \sum_{i=1}^{n} \sum_{j=1}^{n} \left( \frac{1}{\vX_{ij}^{(k)}+\epsilon} \right) [\vX_{ij} - \vX^*_{ij}]_{+}  + \lambda \sum_{i=1}^{n} \sum_{j=1}^{n} \left( \frac{1}{\vX_{ij}^{(k)}+\epsilon} \right) [\vX_{ij} - \vX^*_{ij}]_{-}\\
  \ge &  \lambda \sum_{i=1}^{n} \sum_{j=1}^{n} \left( \frac{1}{a+\epsilon} \right) [\vX_{ij} - \vX^*_{ij}]_{+} + \lambda \sum_{i=1}^{n} \sum_{j=1}^{n} \left( \frac{1}{1-a+\epsilon} \right) [\vX_{ij} - \vX^*_{ij}]_{-}\\
  = & \frac{\lambda}{2} \sum_{i=1}^{n} \sum_{j=1}^{n} \left( \frac{1}{a+\epsilon} \right) |\vX_{ij} - \vX^*_{ij}| 
    - \frac{\lambda}{2} \sum_{i=1}^{n} \sum_{j=1}^{n} \left( \frac{1}{1-a+\epsilon} \right) |\vX_{ij} - \vX^*_{ij}| \\
  = & \frac{(1-2a)\lambda}{2(a+\epsilon)(1-a+\epsilon)} \sum_{i=1}^{n} \sum_{j=1}^{n} |\vX_{ij} - \vX^*_{ij}| ,
\end{align*}
where the first equality follows from Lemma~\ref{lemma}.

Plugging the above equation into \eqref{eq:bound2} yields
\begin{align*}
    &f(\vX) + \lambda \sum_{i=1}^{n} \sum_{j=1}^{n} \left( \frac{1}{\vX_{ij}^{(k)} + \epsilon} \right) \vX_{ij} \ge f(\vX^*) - L\|\vX^* - \vX\|_{\Fsf}\\
&    
   \hspace{1cm} + \lambda \sum_{i=1}^{n} \sum_{j=1}^{n} \left( \frac{1}{\vX_{ij}^{(k)} + \epsilon} \right) \vX^*_{ij} + \frac{(1-2a)\lambda}{2(a+\epsilon)(1-a+\epsilon)} \sum_{i=1}^{n} \sum_{j=1}^{n} |\vX_{ij} - \vX^*_{ij}|\nonumber\\
\geq &f(\vX^*) - L\|\vX^* - \vX\|_{\Fsf} 
    + \lambda \sum_{i=1}^{n} \sum_{j=1}^{n} \left( \frac{1}{\vX_{ij}^{(k)} + \epsilon} \right) \vX^*_{ij} + \frac{(1-2a)\lambda}{2(a+\epsilon)(1-a+\epsilon)} \|\vX^* - \vX\|_{\Fsf} \\
    = &f(\vX^*) + \left( \frac{(1-2a)\lambda}{2(a+\epsilon)(1-a+\epsilon)} - L \right) \|\vX^* - \vX\|_{\Fsf} + \lambda \sum_{i=1}^{n} \sum_{j=1}^{n} \left( \frac{1}{\vX_{ij}^{(k)} + \epsilon} \right) \vX^*_{ij} \\
    \geq &f(\vX^*) + \lambda \sum_{i=1}^{n} \sum_{j=1}^{n} \left( \frac{1}{\vX_{ij}^{(k)} + \epsilon} \right) \vX^*_{ij},   
\end{align*} 
where the second inequality is by $\sum_{i=1}^{n} \sum_{j=1}^{n} |\vX_{ij} - \vX^*_{ij}|\ge \|\vX^* - \vX\|_{\Fsf}$ and the last inequality follows from the condition on $\lambda$. This completes the proof.
\end{proof}
\end{theorem}

\begin{remark}
By Theorem~\ref{mainthm}, we can conclude that if there exists $k_N > 0$ such that $\|\vX^{(k_N)} - \vX^*\|_{\Fsf} < \frac{1}{2}$ and $\lambda_{k_N} > \frac{2(a + \epsilon)(1 - a + \epsilon)}{(1 - 2a)} L$, then $\vX^{(k)}$ converges to $\vX^*$. 
\end{remark}

\section{A Practical Reweighted Algorithm for Graph Matching Problem}
Since $\vX$ update in Algorithm~\ref{alg:LRRAF} is a constraint optimization problem and cannot be explicitly expressed. We present a practical regularization algorithm in this section. Specifically, a sequence of relaxed problems in the form of~\eqref{reweightedsub} is solved by using the project gradient method:
\begin{equation} \label{reweightedsub}
\min_{\vX \in \mathcal{D}_n} F_{\lambda, \vX^{(0)},\epsilon} (\vX)= \min_{\vX \in \mathcal{D}_n} \|\vA\vX - \vX\vB\|_{\mathsf{F}}^2 + \lambda \sum_{i=1}^{n} \sum_{j=1}^{n} \left(\frac{1}{\vX_{ij}^{(0)} + \epsilon}\right) \vX_{ij}.
\end{equation}
The main algorithm is given in Algorithm~\ref{alg:main}. The algorithm with projected gradient subsolver to update $\vX_k$ is given in Algorithm~\ref{alg:PGmethod}. The Projection (onto $\Dcaln$) subsolver is given in Algorithm~\ref{alg:dualBB}.
\begin{algorithm}[h]
\caption{A Practical Reweighted Regularization Algorithm }
\label{alg:main}
\begin{algorithmic}[1]
\STATE \textbf{Input}: $\vA \in \mathbb{R}^{n \times n}$, $\vB \in \mathbb{R}^{n \times n}$, $\epsilon_0 > 0$, and $\lambda_0 > 0$
\STATE \textbf{Initialization}: Set $k = 0$, $\tau > 0$, $\vX_{-1} = \frac{1}{n} \mathbf{1_n} \in \mathcal{D}_n$
\WHILE{not convergent}
    \STATE Set $\vX_k^{(0)}=\vX_{k-1}$
    \STATE Find an approximate minimizer $\vX_{k}$ of problem \eqref{reweightedsub} with $\lambda_k, \vX_k^{(0)}, \epsilon_k$
    \STATE $\lambda_{k+1} \ge \lambda_k$, $\epsilon_{k+1} \le \epsilon_k$
    \STATE $k \gets k + 1$
\ENDWHILE
\end{algorithmic}
\end{algorithm}

\subsection{Projected Gradient Method} 
To conduct a comparable numerical experiment using reweighted regularization, we adopt a similar projected gradient approach as introduced in \cite{WZW}, applying the projected gradient method to solve the minimization problem and update $\vX_{k}$ in Algorithm~\ref{alg:main}. Specifically, at the $k$-th iteration, starting from the initial $X_k^{(0)}$, the projected gradient method for solving problem~\eqref{reweightedsub} with $\lambda_k$, $\vX_k^{(0)}$, and $\epsilon_k$ proceeds as follows:
\begin{equation}
\vX_{k}^{(i+1)} = \vX_k^{(i)} + \delta^j \vD^{(i)}, \quad \delta \in (0,1),
\end{equation}
where the search direction:
\begin{equation}
\vD^{(i)}= \proj_{\Dcaln} \big(\vX_k^{(i)} - \alpha_i \nabla F_{\lambda_k, \vX_k^{(0)}, \epsilon_k} (\vX_k^{(i)})\big) - \vX_k^{(i)}
\end{equation}
and $\proj_{\Dcaln}(\cdot)$ is the projection onto $\Dcaln$. A fast dual gradient is used for computing projection onto $\Dcaln$. The parameter $j$ is the smallest nonnegative integer satisfying the nonmonotone line search condition introduced in~\cite{zhang2004nonmonotone}:
\begin{equation}
\label{equ:nmls}
  F_{\lambda_k,\vX_{k}^{(0)},\epsilon_k}(\vX_k^{(i)} + \delta^j \vD^{(i)}) \leq C_i+ \theta \delta^j \langle \nabla F_{\lambda_k, \vX_{k}^{(0)}, \epsilon_k}(\vX_k^{(i)}), D^{(i)}\rangle, \quad  \theta \in (0,1),
\end{equation}
where the reference function value $C_{i+1}$ is updated as the convex combination of $C_i$ and $F_{\sigma_k, p, \epsilon_k}(X_k^{(i)})$:
\begin{equation}
C_{i+1} =  \frac{\eta Q_i C_i + F_{\lambda_k, \vX_{k}^{(0)}, \epsilon_k} (\vX_{k}^{(i)})}{Q_{i+1}}
\end{equation}
with $Q_{i+1} = \eta Q_i + 1, ~\eta = 0.85,~C_0 = F_{\lambda_k, \vX_{k}^{(0)},
\epsilon_k}(\vX_k^{(0)}),~Q_0 = 1.$ 

\begin{algorithm}
\caption{An Practical Algorithm with Project Gradient Method}
\label{alg:PGmethod}
\begin{algorithmic}[1]
\STATE \textbf{Initialization}: Set $k = 0$, $\tau > 0$, $\vX_{-1} = \frac{1}{n} \mathbf{1_n} \in \mathcal{D}_n$,
$\epsilon_0$, $\lambda_0 \geq 0$, 
$\theta, \delta, \gamma \in (0,1)$.
  \WHILE{$\|{\vX}_{k-1}\|_0 > n + \tau$}
  \STATE Choose $\vX_k^{(0)}=\vX_{k-1}$. Set $i=0$, $\alpha_i>0$
    \WHILE{$\|X_k^{(i)} - X_k^{(i-1)}\|_{\Fsf}/\sqrt{n}>\tau$} 
        \STATE Compute $D^{(i)}= P_{\Dcaln} \big(X_k^{(i)} - \alpha_i \nabla F_{\lambda_k, \vX_k^{(0)}, \epsilon_k} (X_k^{(i)})\big) - X_k^{(i)}.$
        \STATE Find the smallest $j$ such that  $\delta^j$ satisfies \eqref{equ:nmls}
        \STATE Set $X_{k}^{(i+1)} = X_k^{(i)} + \delta^j D_k^{(i)}$
        \STATE $i \gets i + 1$
    \ENDWHILE
    \STATE  $\epsilon_{k+1}=\max(\delta\epsilon_{k},\epsilon_{\min}),$ $\lambda_{k+1}=\min(\lambda_{k}+\gamma,\lambda_{\max})$ 
    \STATE $k \gets k + 1$
    \ENDWHILE
    \end{algorithmic}
\end{algorithm}

\subsection{Fast Dual Gradient Algorithm for Computing the Projection onto $\Dcaln$} \label{subsection:dualBB:proj}
In this subsection, we briefly summarize the fast dual gradient method introduced in~\cite{WZW} for computing the projection onto $\Dcaln$~\eqref{equ:prob:proj}.

\begin{equation}
\label{equ:prob:proj}
\proj_{\Dcaln}(\vC) = \argmin_{\vX \in \mathbb{R}^{n \times n}} \, \frac{1}{2} \|\vX - \vC\|_{\text{F}}^2 \quad \text{subject to} \quad \vX \ve = \ve, \ \vX^{\top} \ve = \ve, \ \vX \geq 0,
\end{equation}
The Lagrangian dual problem of \eqref{equ:prob:proj} is
\begin{equation} \label{equ:prob:proj:dual0}
\max_{\vy, \vz} \, \min_{\vX \geq 0} \mathcal{L} (\vX, \vy, \vz),
\end{equation}
where
\begin{equation}
\mathcal{L}(\vX, \vy, \vz) = \frac{1}{2} \|\vX - \vC\|_{\text{F}}^2 - \langle \vy, \vX \ve - \ve \rangle - \langle \vz, \vX^{\top} \ve - \ve \rangle,
\end{equation}
and $\vy, \vz \in \mathbb{R}^n$ are the Lagrange multipliers of the linear constraints $\vX \ve = \ve$ and $\vX^{\top} \ve = \ve$ respectively.

Let $\proj_{+}(\cdot)$ denote the projection onto the nonnegative orthant. The dual problem \eqref{equ:prob:proj:dual0} can then be equivalently rewritten as:
\begin{equation}
\label{equ:prob:proj:dual}
\min_{\vy, \vz} \, \theta(\vy, \vz) := \frac{1}{2} \|\proj_{+} \left( \vC + \vy \ve^{\top} + \ve \vz^{\top} \right)\|_{\text{F}}^2 - \langle \vy + \vz, \ve \rangle.
\end{equation}
The derivative of $\theta(\vy, \vz)$ can be written as:
\begin{equation}
\nabla \theta(\vy, \vz) = \left[ \begin{array}{l} 
\proj_{+} \left( \vC + \vy \ve^{\top} + \ve \vz^{\top} \right) \ve - \ve \\[4pt] 
\proj_{+} \left( \vC + \vy \ve^{\top} + \ve \vz^{\top} \right)^{\top} \ve - \ve 
\end{array} \right].
\end{equation}
The gradient method using the Block-Broyden step sizes to solve problem \eqref{equ:prob:proj:dual} is outlined in~\ref{alg:dualBB}. 
After obtaining the optimal solution $\vy^*$ and $\vz^*$ of \eqref{equ:prob:proj:dual}, one can recover the projection of $\vC$ by
\begin{equation}
\proj_{\Dcaln}(\vC) = \proj_{+} \left( \vC + \vy^* \ve^{\top} + \ve (\vz^*)^{\top} \right).
\end{equation}

\begin{algorithm}
\caption{Dual Block-Broyden Method for Projection (Dual BB)}
\label{alg:dualBB}
\begin{algorithmic}[1]
\STATE \textbf{Input:} Matrix $\vC$, tolerance $tol$
\STATE Initialize: $\vy = 0$, $\vz = 0$, $\ve = \mathbf{1}$ (all ones vector)
\STATE Set initial learning rate: $\alpha = 0.01$
\STATE Initialize previous gradients and iterates: $\nabla{\theta_\vy}_{0} = 0$, $\nabla{\theta_\vz}_{0} = 0$, $\vy_0 = \vy$, $\vz_0 = \vz$
\WHILE{{$\| \begin{bmatrix} \nabla{\theta_\vy} \\ \nabla{\theta_\vz}\end{bmatrix} \|_\Fsf > tol$}}
    \STATE Calculate $\vM = \vC + \vy \cdot \ve^\top + \ve \cdot \vz^\top$
    \STATE Calculate projection: $\vM_+ = \max(\vM, 0)$
    \STATE Compute gradients: $\nabla{\theta_\vy} = \vM_+\cdot \ve - \ve$, $\nabla{\theta_\vz} = \vM_+^\top \ve - \ve$
    \STATE Compute step sizes for $\vy$ and $\vz$: $\vs_\vy = \vy - \vy_0$, $\vs_\vz = \vz - \vz_0$
    \STATE Update previous iterates: $\vy_0 = \vy$, $\vz_0 = \vz$
    \STATE Compute gradient differences: $\delta_\vy = \nabla{\theta_\vy}-\nabla{\theta_\vy}_{0}$, $\delta_\vz = \nabla{\theta_\vz}-\nabla{\theta_\vz}_{0}$
    \STATE Compute BB step sizes: 
    $\alpha_\vy = \frac{\vs_\vy^\top \vs_\vy}{\vs_\vy^\top \delta_\vy}$, $\alpha_\vz = \frac{\vs_\vz^\top \vs_\vz}{\vs_\vz^\top \delta_\vz}$
    \STATE Set $\alpha = 0.5 \cdot (\alpha_\vy + \alpha_\vz)$
    \STATE Update iterates: $\vy = \vy - \alpha\nabla{\theta_\vy} $, $\vz = \vz - \alpha \nabla{\theta_\vz}$
    \STATE Store current gradients for next iteration: $\nabla{\theta_\vy}_{0}= \nabla{\theta_\vy}$, $\nabla{\theta_\vz}_{0} = \nabla{\theta_\vz}$
\ENDWHILE
\STATE \textbf{Output:} Optimal values $\vy^* = \vy$, $\vz^* = \vz$
\end{algorithmic}
\end{algorithm}

\section{Application}
\subsection{Network Alignment Review}
The quadratic programming formulation of a network alignment objective is given in~\cite{Ravindra2019RigidGA}. Specifically, given two undirected graphs $\vA = G(\vV_A, \vE_A)$ and $\vB = G(\vV_B, \vE_B)$, with vertex sets $\vV_A$ and $\vV_B$ of size $|\vV_A|$ and $|\vV_B|$, the goal of network alignment is to find a matching $\mathcal{M}$ between the vertices using a prior knowledge matrix $\vL$ that encodes the likelihood of vertex alignments. A binary matrix $\vX$ is introduced to represent the matching, where 
$\vX_{ij} = 1$ indicates a match between vertex $i$ in $\vA$ and vertex $j$ in $\vB$. Then, the  corresponding quadratic program can be formulated as
\begin{equation}
\label{equ:network}
 \max_{\vX} \, \alpha\vL\cdot\vX + \beta \vA\cdot \vX\vB\vX^\top, 
\end{equation}
subject to the constraints,
\vspace{-0.2cm}
\begin{equation}
\quad \sum_{i=1}^{|\vV_A|} \vX_{ij} \leq 1, \quad \forall j = 1, \dots,|\vV_A|,~\vX_{ij} \in \{0, 1\},
\end{equation}
\vspace{-0.2cm}
\begin{equation}
\quad \sum_{j=1}^{|\vV_B|} \vX_{ij} \leq 1, \quad \forall i = 1, \dots, |\vV_B|,~\vX_{ij} \in \{0, 1\},
\end{equation}
where $\vA \cdot \vX\vB\vX^T = \sum_{i=1}^{|\vV_A|}\sum_{j=1}^{|\vV_B|} \vA_{ij}(\vX\vB\vX^\top)_{ij}$ is called overlap of matching $\mathcal{M}$. 
Here, $\alpha$ and $\beta$ are non-negative constants that allow for tradeoff between matching weights from the prior and the number of overlapping edges.

In our numerical experiments, we consider the matrix $\vA$ and $\vB$ as adjacency or distance matrices for undirected and weighted graph with $|\vV_A|=|\vV_B|=n$, $\vX_{ij}\in\Pi_n$ and no prior knowledge is given (i.e. $\vL=0$). Notice that when $\vX \in \Pi_n$,
\begin{equation}\|\vA \vX - \vX \vB\|_{\Fsf}^2 = -2\,\vA\cdot\vX\vB\vX^T + \|\vA\|_{\Fsf}^2 + \|\vB\|_{\Fsf}^2.
\end{equation}
Hence, problem~\eqref{equ:grahpmatch} is equivalent to~\eqref{equ:network} when $\vL=0$, $|\vV_A|=|\vV_B|=n$. 

\subsection{Graph in Social Network Alignment Problem}
The social network alignment problem aims to identify individuals in two different graphs who share similar connection patterns. Given two social network graphs (e.g. Figure~\ref{fig:social}), this task can be completed by solving \eqref{equ:network}, where the matrices $\vA$ and $\vB$ represent the distance matrices for each social network graph. Specifically, each element in the distance matrix corresponds to the shortest distance between a pair of nodes in the graph.
The distance matrix for a graph can be computed using the Breadth-First Search (BFS) algorithm~\cite{bundy1984breadth}. This algorithm operates by growing a tree from a given node, expanding outward while incrementing the hop count at each expansion step and nodes that have already been visited are ignored.

\subsection{Graph in Shape correspondence problem}
Given two manifolds $\mathcal{M}_1$ and $\mathcal{M}_2$ sampled by point clouds $P_1=\{x_i\}_{i=1}^n$ and $P_2 = \{ y_i \}_{i=1}^n$, the task of dense shape correspondence is to find a point-to-point map between $P_1$ and $P_2$. The task can be achieved by solving graph matching problem \eqref{equ:grahpmatch} where $\vA \in \mathbb{R}^{n \times n}$ and $\vB \in \mathbb{R}^{n \times n}$ are two pairwise descriptors, such as geodesic distances, between points in $P_1$ and $P_2$ (e.g. Figure~\ref{fig:shape}). 

From a more practical view, the surface $\mathcal{M}$ can be discretized using a triangular mesh $T=\{\tau_l\}_{i=1}^n$ with edges $E = \{e_{ij}\}$, and vertices of the mesh are denoted by $V=\{\vx_i\}^{n}_{i=1}$. For each edge $e_{ij}$ connecting vertices $p_i$ and $p_j$, the angles opposite to the edge are defined as $\alpha_{ij}$ and $\beta_{ij}$. The stiffness matrix $\vA_s$ is given by~\cite{reddy1993introduction,reuter2009discrete}:
\begin{equation}
\vA_{s_{ij}} = 
\begin{cases} 
-2 \left( \cot \alpha_{ij} + \cot \beta_{ij} \right) & \text{if } i \sim j \\
\sum_{k\sim i}\vA_s(i,j) & \text{if } i = j,
\end{cases}
\end{equation}
where $i \sim j$ indicates that $i$ and $j$ are connected by an edge. For more information on sparse pairwise descriptors, we refer the readers to~\cite{xiang2020efficient}.

\begin{figure}[h!]
\begin{center}
\begin{minipage}{0.43\textwidth}
\centering
\includegraphics[width=\textwidth]{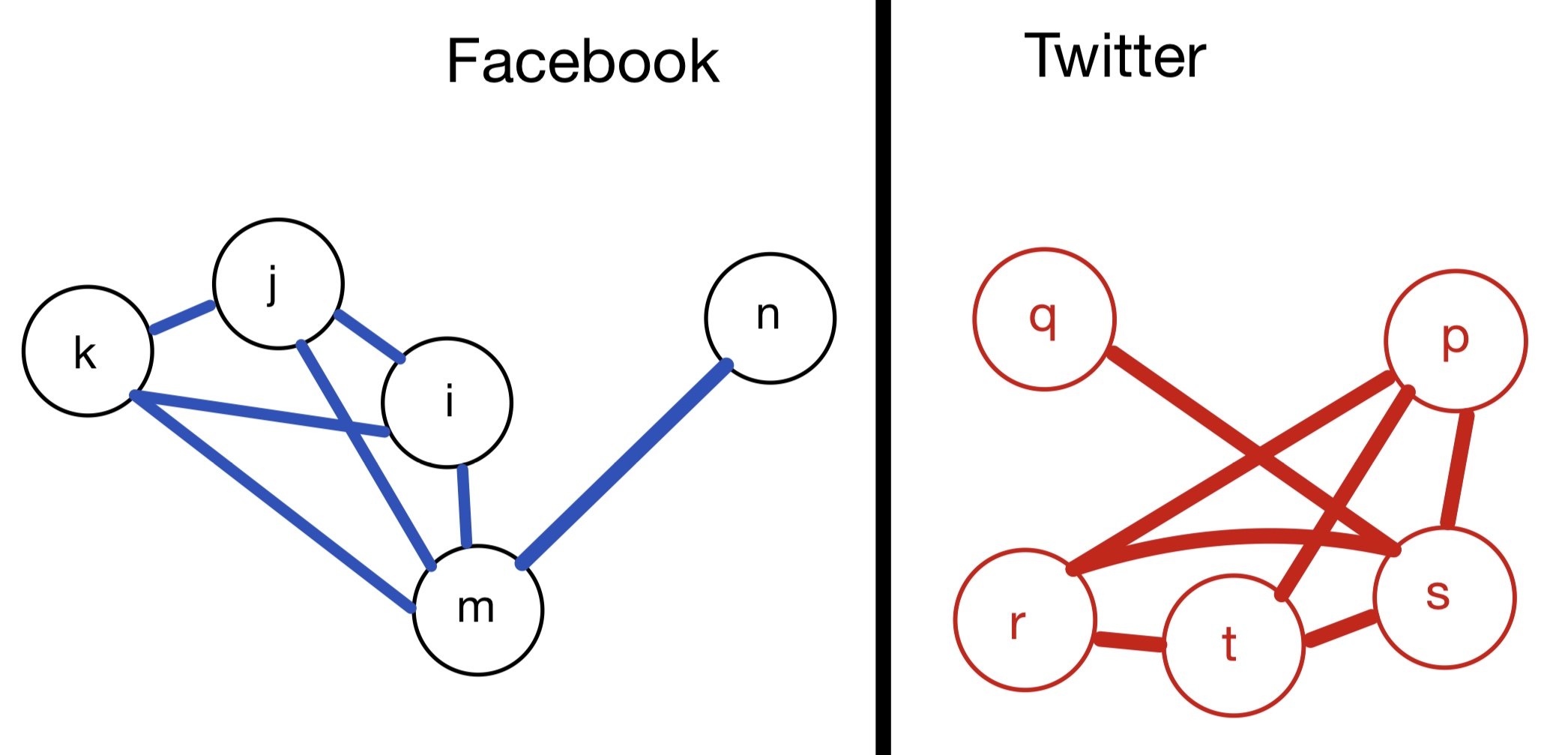}
\caption{An illustration of the graph matching problem in social network alignment. The two graphs represent connections between individuals on different platforms. Facebook (left) and Twitter (right). The goal is to identify corresponding nodes (e.g. individuals) between the two graphs based on their structural connectivity.}
\label{fig:social}
\end{minipage}%
\hspace{0.05\textwidth} % Adds horizontal space between the two images
\begin{minipage}{0.43\textwidth}
\centering
\includegraphics[width=\textwidth]{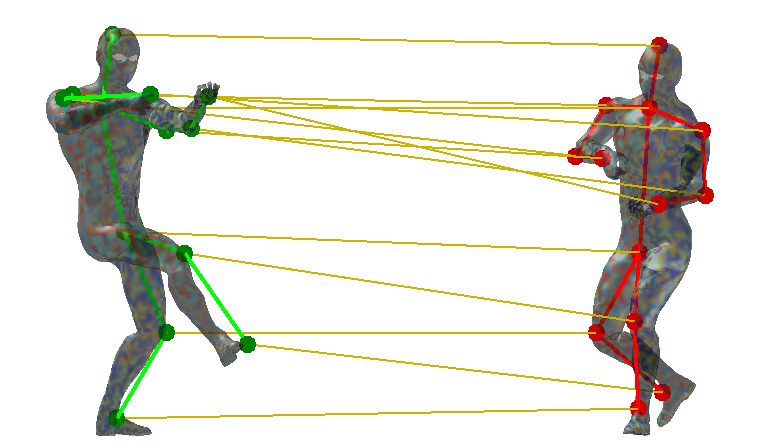}
\vspace{5pt} % Adds vertical space between caption and image
\caption{The body matching result visualization. Matrices $\vA$ and $\vB$ represent the pairwise geodesic distances between joints in two individuals. The visualization is obtained by solving relaxed graph matching problem with linear reweighted regularization term that we proposed.}
\label{fig:shape}
\end{minipage}
\end{center}
\end{figure}

\section{Experimental results}
\subsection{Experimental Dataset}
\label{subsec}
In our experiments, matrix $\vA$ in~\eqref{equ:grahpmatch} is obtained by first creating $n$ random 2D points using \verb|rand(n, 2)*10| in MATLAB and then each entry $\vA_{ij}$ is computed as the Euclidean distance between points $i$ and $j$. Matrix $\vB$ = $\vP*\vA*\vP'+\vC$, where $\vP$ is randomly generated permutation matrix and $\vC$ is generated by 2D points using \verb|rand(n, 2)*0.5|, and then each entry $\vC_{ij}$ is computed as the Euclidean distance between points $i$ and $j$. 
\subsection{Implementation Details}
To compare our proposed method with $L_p$ norm regularization, we implemented $L_p$-regularization-based solvers for the graph matching problem using $L_{p=0.75}$ and $L_{p=0.5}$ regularization terms as $h(\vX)$, following Algorithm 2 in~\cite{WZW}. We applied the linear reweighted regularization term using Algorithm~\ref{alg:PGmethod}, where the projection is computed using Algorithm~\ref{alg:dualBB}. The initial point for all regularization algorithms is chosen as $\vX_0 = \frac{1}{n} \boldsymbol{1_n}$, with the initial $\epsilon_0 = 1$, safeguards $\epsilon_{\min} = 10^{-3}$, and $\lambda_{\max} = 10^6$. The shrinkage parameter for updating $\lambda_k$ is set to $\gamma = 0.9$. We define $r(\vX) = \|\vA\vX - \vX\vB\|_{\Fsf}^2$, and denote the best numerical solution obtained at real time $t$ as $\vX_k$. Since only a relatively small perturbation $\vC$ is introduced into the permuted matrix $\vB$ in Section~\ref{subsec}, we assume $\vX^* = \argmin{\vX \in \Pi_n} r(\vX) = \vP$, where $\vP$ is the permutation matrix generated in Section~\ref{subsec}. The objective error at real time $t$ is defined as $|r(\vX_k) - r(\vP)|$, and the residual at real time $t$ is given by $\|\vX_k - \vP\|_{\Fsf}$. The experimental results using different regularization terms are presented in Figures~\ref{fig:obj50} and~\ref{fig:res50}.

\begin{figure}[h!]
\begin{center} 
\includegraphics[width=0.7\textwidth]{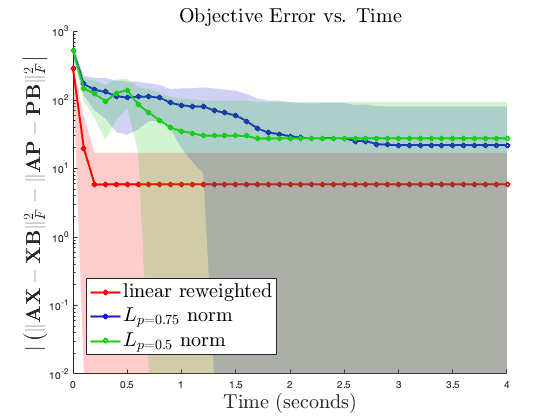}
\end{center}
\caption{Average and standard deviation results of Objective Error by linear reweighted, $L_{p=0.75}$ and $L_{p=0.5}$ regularization term, on solving 50 independently random generated instances of graph matching problem of dimension n = 50.
}\label{fig:obj50}
\end{figure}

\begin{figure}[h!]
\begin{center} 
\includegraphics[width=0.7\textwidth]{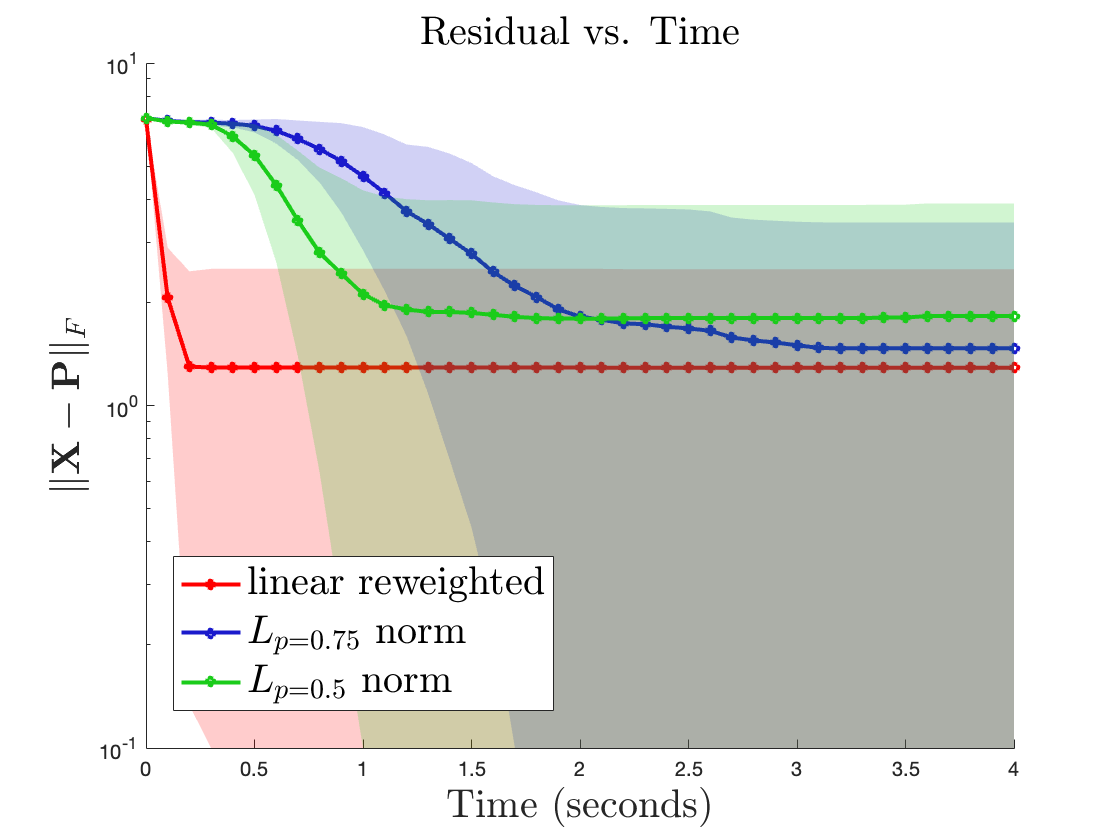}
\end{center}
\caption{Average and standard deviation results of Residual by linear reweighted, $L_{p=0.75}$ and $L_{p=0.5}$ regularization term, on solving 50 independently random generated instances of graph matching problem of dimension n = 50.
}\label{fig:res50}
\end{figure}

\section{Discussion}
In this work, we established the theory behind the linear reweighted regularization term. We provided a comprehensive justification for the use of linear reweighted regularization in solving the relaxed graph matching problem. Additionally, we conducted numerical experiments to empirically demonstrate that our reweighted regularization framework outperforms other regularization terms. In our future work, we plan to extend the graph matching problem to more applicable topics and focus on large-scale real-world data. We also aim to propose an Augmented Lagrange Multiplier (ALM)-based algorithm to ensure the effectiveness of our reweighted regularization term when facing large-scale data.

\section*{Acknowledgments}
The first author would like to thank Ms. Yueshan Ai for validating the proof of Theorem~\ref{mainthm}.
\subsection*{Conflict of interest declaration}
The author(s) has/have no competing interests to declare.
\subsection*{Data Availability}
The codes and source data files are available on GitHub at 

\href{https://github.com/rongxuan-li/graph-match/blob/main/README.md}{https://github.com/rongxuan-li/graph-match}
\bibliographystyle{siam}
\bibliography{Main}
\end{document}